\documentclass{article}

\usepackage{amsmath}
\usepackage{amsfonts}
\usepackage{amssymb}
\usepackage{amsthm}
\usepackage{mathtools}
\usepackage{graphicx}
\usepackage{color}
\usepackage{fancyhdr}
\usepackage{geometry}
\usepackage{hyperref}
\usepackage{mathrsfs}
\usepackage[nottoc,numbib]{tocbibind}

\usepackage[utf8]{inputenc}

\allowdisplaybreaks

\makeatletter
\DeclareRobustCommand{\format@sec@number}[2]{{\normalfont\upshape#1}#2}

\makeatother

\def\a{\mathbf a}
\def\b{\mathbf b}

\def\R{\mathbb R}
\def\N{{\mathbb N}}

\def\Z{\mathbb Z}

\def\S{\mathbb S}

\def\Q{\mathbb Q}

\def\({\biggl(}
\def\){\biggr)}
\def\<{\mathbf{\langle}}
\def\>{\mathbf{\rangle}}

\def\diff{\text{Diff}}

\newcommand{\Meng}[2]{\left\{#1\mathrel{}\middle|\mathrel{}#2\right\}}
\newcommand{\abs}[1]{\left\lvert#1\right\rvert}
\newcommand{\norm}[1]{\left\lVert#1\right\rVert}
\newcommand{\vertiii}[1]{{\left\vert\kern-0.25ex\left\vert\kern-0.25ex\left\vert #1
		\right\vert\kern-0.25ex\right\vert\kern-0.25ex\right\vert}}

\numberwithin{equation}{section}

\numberwithin{equation}{section}

\newtheorem{theorem}[equation]{Theorem}
\newtheorem{prop}[equation]{Proposition}
\newtheorem{lemma}[equation]{Lemma}
\newtheorem{cor}[equation]{Corollary}

\newtheorem{maintheorem}{Theorem}

\theoremstyle{definition}
\newtheorem{definition}[equation]{Definition}

\theoremstyle{definition}

\theoremstyle{remark}

\theoremstyle{plain}

\newtheorem*{que}{Question}

\title{There is no complete numerical invariant for smooth conjugacy of circle diffeomorphisms}
\author{Philipp Kunde \footnote{University of Hamburg, Department of Mathematics, Bundesstraße 55, 20146 Hamburg}}
\date{}

\graphicspath{{pic/}}

\begin{document}
	
	\maketitle
	
	\begin{abstract} 
		Classical results by Poincar\'e and Denjoy show that two orientation-preserving $C^2$ diffeomorphisms of the circle are topologically conjugate if and only if they have the same rotation number. We show that there is no possibility of getting such a complete numerical Borel invariant for the conjugacy relation of orientation-preserving circle diffeomorphisms by homeomorphisms with higher degree of regularity. For instance, we consider conjugacy by Hölder homeomorphisms or by $C^k$-diffeomorphisms with $k\in \mathbb{Z}^+ \cup \{\infty\}$. The proof combines techniques from Descriptive Set Theory and a quantitative version of the Approximation by Conjugation method for circle diffeomorphisms.
	\end{abstract}

		\insert\footins{\footnotesize - 
		}

	\section{Introduction}\label{section introduction}
	A fundamental theme in dynamics is the classification of systems up to appropriate equivalence relations. For instance, the equivalence relation of \emph{topological conjugacy} preserves the qualitative behavior of topological dynamical systems. Here, two continuous maps $T:X \to X$ and $S: Y \to Y$ of compact metric spaces $X, Y$ are called topological conjugate (or topologically equivalent) if there is a homeomorphism $h: X \to Y$ such that $S=h \circ T \circ h^{-1}$. Smale's celebrated program proposes to classify topological or smooth dynamical systems up to topological conjugacy.
	
	For the purpose of classification, one seeks \emph{dynamical invariants} that are easy to compute and help determine whether two systems can be equivalent to each other. In the best case scenario, these invariants can be expressed as a single number. While one often can find invariants that are preserved under equivalence, even within specific classes of dynamical systems it turns out to be very hard to find \emph{complete invariants}, that is, invariants that agree for two systems if and only if the systems are equivalent to each other. Indeed, having complete numerical invariants is one of the key notions of classifiability.
	
	One well-studied numerical invariant is the \emph{rotation number} $\tau(f)$ for an orientation-preserving circle homeomorphism $f$ (see Section \ref{subsec:circle} for its definition). We denote the collection of orientation-preserving homeomorphisms of the circle by $\mathcal{H}$. Then the rotation number is an invariant for the equivalence relation on $\mathcal{H}$ of conjugacy by an orientation-preserving homeomorphism. In the converse direction, Poincar\'e showed in 1885 that any $f\in \mathcal{H}$ with an irrational rotation number $\tau(f)$ and a dense orbit is topologically conjugate to $R_{\tau(f)}$ (see \cite[Theorem 11.2.7]{HK}).	Thus, homeomorphisms of the circle with a dense orbit present a success story of a classification by a numerical invariant. Furthermore, Denjoy showed that $C^1$ circle diffeomorphisms with irrational rotation number and derivative of bounded variation have a dense orbit and, hence, are completely classified by their rotation number (see \cite{De} or \cite[Theorem 12.1.1]{HK}).
	
	In this paper, we focus on the conjugacy relation of orientation-preserving circle diffeomorphisms by homeomorphisms with higher degree of regularity. The question of the smoothness of the conjugacy of circle diffeomorphisms to a rotation is a classical problem in dynamical systems. Arnold~\cite{Ar} proved a \emph{local rigidity} result that every analytic circle diffeomorphism with a Diophantine\footnote{A number $\alpha\in \R$ is called Diophantine of class $D(\nu)$ for $\nu \geq 0$, if there exists $C > 0$ such that 
		\[
		\abs{\alpha -\frac{p}{q}} \geq \frac{C}{q^{2+\nu}} \ \text{ for every $p\in \Z$ and $q\in \N$}.
		\]
		A number $\alpha$ is called \emph{Diophantine} if it is in $D(\nu)$ for some $\nu \geq 0$.} rotation number $\tau(f)$, that is sufficiently close to the rigid rotation $R_{\tau(f)}$, is analytically conjugate to $R_{\tau(f)}$. A celebrated \emph{global rigidity} result by M.\ Herman \cite{He} shows that every $C^{\infty}$ (respectively, $C^{\omega}$) diffeomorphism of the circle with Diophantine rotation number $\tau(f)$ is $C^{\infty}$-conjugate (respectively, $C^{\omega}$-conjugate) to $R_{\tau(f)}$. Papers by Yoccoz \cite{Y} and Katznelson-Ornstein \cite{KO} give further insight into the dependence of the smoothness of the conjugacy on the Diophantine properties of the rotation numbers: If a $C^k$ circle diffeomorphism has a rotation number $\tau(f) \in D(\nu)$ with $k >\nu +2$, then $f$ is $C^{k-1-\nu-\varepsilon}$-conjugate to $R_{\tau(f)}$ for every $\varepsilon >0$.
	
	In the negative direction, Arnold provided an example of a real-analytic circle diffeomorphism $f$ that is conjugate to rotation $R_{\tau(f)}$ by a non-differentiable homeomorphism (see \cite{Ar} or \cite[Theorem 12.5.1]{HK}). In \cite[section 12.6]{HK} further constructions of $C^{\infty}$ circle diffeomorphisms with conjugacies of intermediate regularity are presented. For example, for every $r \in \N$ there is a $C^{\infty}$ circle diffeomorphism $f$ that is conjugate to $R_{\tau(f)}$ via a conjugacy that is $C^r$ but not $C^{r+1}$. These constructions are obtained by the \emph{Approximation by Conjugation method} (also called AbC method or Anosov-Katok method) that was introduced in the influential paper \cite{AK}. Here, diffeomorphisms are constructed inductively as limits of conjugates $T_n = H_n \circ R_{\alpha_{n+1}} \circ H^{-1}_n$, where at each step $n$ one updates the conjugation maps $H_n=H_{n-1}\circ h_n$  as well as the numbers $\alpha_{n+1} \in \Q$ close to the previous number $\alpha_n$. We refer to the survey papers \cite{FK} and \cite{Ksurvey} for expositions of the AbC method and its wide range of applications in dynamics. 
	
    Refined versions of the results from \cite[section 12.6]{HK} are obtained in \cite{M} and \cite{KuCircle}, where diffeomorphisms with conjugacies of intermediate regularity and prescribed Liouville\footnote{An irrational number $\alpha$ is called \emph{Liouville} if it is not Diophantine, that is, for every $C>0$ and every $n \in \N$ there are infinitely many pairs $p\in \Z$, $q \in \N$ such that
    	\[
    	0<\abs{\alpha - \frac{p}{q}}<\frac{C}{q^n}.
    	\]}
     rotation number are constructed. One of the key ingredients in these AbC constructions are precise norm estimates on the conjugation maps. In the literature, such versions of the AbC method go by the name of \emph{quantitative constructions} (see e.g. \cite[section 2]{Ksurvey}) which were initiated by B.~Fayad and M.~Saprykina in \cite{FS}.   
    
    All these examples show that the rotation number is not a complete numerical invariant for the conjugacy relation of orientation-preserving circle diffeomorphisms by homeomorphisms with higher degree of regularity. They leave open the possibility of a different complete numerical invariant. We build on the constructions in \cite{M} and \cite{KuCircle} and show that there is no complete numerical Borel invariant for the conjugacy relation of circle diffeomorphisms by homeomorphisms with higher degree of regularity. To state our result, for some degree $(D)$ of regularity from Hölder to $C^{\infty}$ let $\mathcal{D}$ be the collection of orientation-preserving circle homeomorphisms with regularity $(D)$. We say that two circle diffeomorphims $S$ and $T$ are $\mathcal{D}$-conjugate if there is a $h\in \mathcal{D}$ such that $S=h\circ T \circ h^{-1}$. Then we look at the equivalence relation $E_{\mathcal{D}}$ on orientation-preserving $C^{\infty}$ circle diffeomorphisms defined by $SE_{\mathcal{D}}T$ if and only if $S$ and $T$ are $\mathcal{D}$-conjugate. For instance, our result holds for the relation of conjugacy by $C^1$ diffeomorphisms.
	
	\begin{maintheorem} \label{theo:A}
		Let $(D)$ be some degree of regularity from Hölder to $C^{\infty}$ and $\mathcal{D}$ be the collection of orientation-preserving circle homeomorphisms with regularity $(D)$. Then there is no complete numerical Borel invariant for $\mathcal{D}$-conjugacy of orientation-preserving $C^{\infty}$ diffeomorphisms of the circle. 
	\end{maintheorem}
    
    In fact, we show for every Liouville number $\alpha$ that there is no complete numerical Borel invariant for $\mathcal{D}$-conjugacy of orientation-preserving $C^{\infty}$ circle diffeomorphisms with rotation number $\alpha$. A key ingredient of our proof is the concept of a \emph{reduction} from the equivalence relation $E_0$ to our conjugacy relations. We explain this concept and further basic tools from descriptive set theory in Section \ref{subsec:Descriptive}. 
    
    In general, the interplay between descriptive set theory and dynamical systems has been fruitful to establish several \emph{anti-classification results} in dynamics which demonstrated in a rigorous way that classification is not possible. A pioneering work \cite{Fe74} by J.~Feldman showed that there cannot exist a complete numerical invariant for measure-theoretic isomorphism of ergodic transformations or even of $K$-automorphisms. In their landmark paper \cite{FRW}, M.~Foreman, D.~Rudolph and B.~Weiss showed that the measure-isomorphism relation for ergodic transformations is not a Borel set when viewed as a subset of $\mathcal{E} \times \mathcal{E}$, where $\mathcal{E}$ denotes the collection of ergodic transformations. Informally speaking, this result says that determining isomorphism between ergodic transformations is inaccessible to countable methods that use a countable amount of information. Recently, Foreman and Weiss \cite{FW3} generalized the aforementioned anti-classification result from \cite{FRW} to the $C^{\infty}$ category by proving that the measure-isomorphism relation among pairs of volume-preserving ergodic $C^{\infty}$-diffeomorphisms of on compact surfaces admitting a non-trivial circle action is not a Borel set with respect to the $C^{\infty}$-topology. In joint work with S.~Banerjee we are even able to show this anti-classification result for real-analytic diffeomorphisms of the $2$-torus \cite{BK2}. Hence, von Neumann's classification problem for the isomorphism relation in ergodic theory is impossible even when restricting to real-analytic diffeomorphisms of $\mathbb{T}^2$. M.~Gerber and the author obtained analogous results for the Kakutani equivalence relation \cite{GeKu}. Recently, Foreman and A.~Gorodetski proved anti-classification results in the context of topological equivalence. For any compact manifold $M$ of dimension at least $2$, Foreman and Gorodetski showed that there is no Borel function from the $C^{\infty}$ diffeomorphisms to the reals
    that is a complete invariant for topological equivalence \cite[Theorem 2]{FG}. Moreover, if dim$(M)\ge5$, they proved that the set of topologically equivalent pairs of diffeomorphisms is not Borel \cite[Theorem 1]{FG}. We refer to the survey article \cite{Fsurvey} by Foreman for an overview of complexity results of structure and classification of dynamical systems. Problem 1 in \cite{Fsurvey} asks if there are complete numerical invariants for orientation-preserving diffeomorphisms of the circle up to conjugation by orientation-preserving diffeomorphisms. Our Main Theorem \ref{theo:A} answers this question in the negative. 
    
    The conjugacy in Theorem~\ref{theo:A} could be either a singular or an absolutely continuous map. In Section~\ref{subsec:Abs} we show that both cases can be realized in the setting of our Main Theorem~\ref{theo:A}. In Section~\ref{sec:Higher} we obtain another generalization of our Main Theorem~\ref{theo:A} to the setting of higher rank actions: we show that there is no complete numerical Borel invariant for $\mathcal{D}$-conjugacy of free $\Z^d$ actions by orientation-preserving $C^{\infty}$ diffeomorphisms of the circle.
    
    Our Main Theorem \ref{theo:A} does not rule out a classification along the lines of the Halmos--von Neumann
    classification of ergodic transformations with discrete spectrum by countable subgroups of the unit circle. To exclude such a classification, it suffices to show that the respective conjugacy relation is not reducible to an $S_{\infty}$-action. Hjorth devised the concept of \emph{turbulence} as a mechanism for showing that an equivalence relation is not reducible to an $S_{\infty}$-action (see \cite[section 5.5]{Fsurvey} for details).
    
    \begin{que}
    	Is the smooth conjugacy relation of smooth circle diffeomorphisms turbulent?
    \end{que}
        
    As mentioned before, the Borel/non-Borel distinction is an important benchmark for classification problems. Hence, it is a natural question if the smooth conjugacy relation of smooth circle diffeomorphisms is Borel or not.

	\begin{que}
		Is the smooth conjugacy relation of smooth circle diffeomorphisms complete analytic?
	\end{que}

\section{Preliminaries}

\subsection{Homeomorphisms of the circle} \label{subsec:circle}
Let $\mathbb{S}^1= \R / \Z$ be the unit circle with rotations $R_{\alpha}:\mathbb{S}^1 \to \mathbb{S}^1$,  $R_{\alpha}(x) = x+\alpha \mod 1$. We denote the collection of orientation-preserving homeomorphisms of $\S^1$ by $\mathcal{H}$. For $x \in \R$ we write $[x]=\pi(x) \in \S^1$, where $\pi: \R \to \S^1$ is the quotient map. Hereby, we can define a \emph{lift} of $f\in \mathcal{H}$ as an increasing continuous function $F: \R \to \R$ with $f \circ \pi = \pi \circ F$, that is, $[F(x)]=f([x])$. Using this terminology we can introduce the rotation number.

\begin{definition} \label{def:RotationNumber}
	Let $f \in \mathcal{H}$ and $F$ be a lift of $f$. Define
	\[
	\tau (F) = \lim_{n \to \infty} \frac{F^n(x)-x}{n}.
	\]
	Then $\tau(f) := [\tau(F)]$ is called the \emph{rotation number} of $f$.
\end{definition}

The definition is justified since $\tau(F)$ exists for all $x \in \R$, is independent of $x$, and we have $[\tau(F_1)]=[\tau(F_2)]$ for all lifts $F_1,F_2$ of $f$ \cite[Proposition 11.1.1]{HK}. Furthermore, one can show that the rotation number is an invariant for conjugacy by orientation-preserving homeomorphisms of the circle, that is, $\tau(f)=\tau (hfh^{-1})$ for all $h\in \mathcal{H}$ \cite[Proposition 11.1.3]{HK}. We also note that $\tau(R_{\alpha})=\alpha$. Hence, the question of the rotation number being a complete numerical invariant can be restated as asking whether $\tau(f)=\alpha$ implies that $f$ is conjugate to $R_{\alpha}$. If $f\in \mathcal{H}$ is topologically conjugate to an irrational rotation, then the conjugacy satisfies the following uniqueness property.

\begin{lemma}\label{lem:ConjRot}
	Let $f \in \mathcal{H}$ be topologically conjugate to an irrational rotation. Then the conjugating homeomorphism $h\in \mathcal{H}$ is unique up to a rotation, that is, if $h_i \circ f = R_{\tau(f)} \circ h_i$ for $i=1,2$, then $h_1 \circ h^{-1}_2$ is a rotation. \\
	In particular, there is a unique conjugating homeomorphism $h\in \mathcal{H}$ satisfying $h(0)=0$.
\end{lemma}
This lemma also shows that the regularity of the conjugacy to the irrational rotation is uniquely determined.

\begin{proof}
	For the reader's convenience we provide a proof of this folklore result (see also \cite[Proposition 3.3.3]{He}). We start with the following observation.\\	
	{\bf Claim: } If a circle homeomorphism $h$ commutes with an irrational rotation $R_{\alpha}$, then $h$ is a rotation. \\
	{\bf Proof: } From $h\circ R_{\alpha} = R_{\alpha} \circ h$ we also get $h \circ R^n_{\alpha} = R^n_{\alpha} \circ h$ for all $n \in \mathbb{Z}$, that is,
	\[
	h(x+n \alpha) = h(x)+n\alpha \ \text{ for all $x \in \mathbb{S}^1$ and $n \in \mathbb{Z}$}.
	\] 
	Since $\alpha$ is irrational, the orbit $\Meng{n\alpha}{n\in \mathbb{Z}}$ lies dense in $\mathbb{S}^1$. Hence, we obtain 
	\[
	h(x+y) = h(x)+y \ \text{ for all $x,y \in \mathbb{S}^1$} 
	\]
	using continuity of $h$. In particular, taking $x=0$ yields
	\[
	h(y)=y + h(0)\ \text{ for all $y \in \mathbb{S}^1$},
	\]
	that is, $h$ is a rotation by $h(0)$. \qed
	
	In the next step, we suppose that 
	\[
	h^{-1}_1 \circ R_{\tau(f)} \circ h_1 = f = h^{-1}_2 \circ R_{\tau(f)} \circ h_2.
	\]
	Hence, $h_1\circ h^{-1}_2$ commutes with $R_{\tau(f)}$. Since $\tau(f)$ is irrational by assumption of the lemma, we can apply the aforementioned Claim and conclude that $h_1\circ h^{-1}_2$ is a rotation.
\end{proof}

\begin{lemma}\label{lem:ConjUniq}
	Let $S , T \in \mathcal{H}$ be topologically conjugate to an irrational rotation $R_{\alpha}$. Then there is a unique $g\in \mathcal{H}$ with $g\circ S \circ g^{-1}=T$ and $g(0)=0$. 
\end{lemma}

\begin{proof}
We consider any $g_1,g_2 \in \mathcal{H}$ satisfying $g_1(0)=0=g_2(0)$ and
\[
g_1\circ S \circ g^{-1}_1 = T = g_2\circ S \circ g^{-1}_2.
\]	
Furthermore, let $h\in \mathcal{H}$ be a conjugacy between $T$ and $R_{\alpha}$. Then we obtain
\[
h \circ g_1 \circ S \circ (h \circ g_1)^{-1} = h \circ T \circ h^{-1} = R_{\alpha} = h \circ g_2 \circ S \circ (h \circ g_2)^{-1}.
\]
By Lemma \ref{lem:ConjRot} we have
\[
h\circ g_1 \circ (h \circ g_2)^{-1} = R_{\beta}
\] 
for some $\beta \in \mathbb{S}^1$. Hence,
\begin{equation} \label{eq:conj}
	g_1 \circ g^{-1}_2 = h^{-1} \circ R_{\beta} \circ h.
\end{equation}
Then the condition $g_1(0)=0=g_2(0)$ implies $\beta = 0$. Using $\beta=0$ in equation \eqref{eq:conj} gives $g_1=g_2$. 
\end{proof}

\subsection{Diffeomorphisms of the circle}
For $k\in \N \cup \{\infty, \omega\}$ let $\mathcal{H}^k$ be the collection of orientation-preserving $C^k$-diffeomorphisms of $\S^1$. Here, the case $k=0$ corresponds to $\mathcal{H}^0 = \mathcal{H}$. Furthermore, for $0<\beta<1$ we denote by $\mathcal{H}^{k+\beta}$ the collection of orientation-preserving $C^k$-diffeomorphisms of $\S^1$ with $\beta$-Hölder continuous $k$-th derivative. 

\begin{definition}
	Let $k \in \N$. We define the $C^k$ norm on $C^k$ functions $f$ of $\mathbb{S}^1$ as
	\[
	\norm{f}_k = \max_{0 \leq i \leq k} \max_{x \in \mathbb{S}^1} \abs{D^if(x)}. 
	\]
	For $C^k$ diffeomorphisms $f,g$ we also define
	\begin{align*}
		\vertiii{f}_k & = \max \{ \norm{f}_k , \, \norm{f^{-1}}_k \}, \\
		d_k(f,g) & = \max \{ \norm{f-g}_k , \, \norm{f^{-1}-g^{-1}}_k \}.
	\end{align*}
\end{definition}

Obviously, $d_k$ describes a metric on $\diff^k(\mathbb{S})$ measuring the distance between the diffeomorphisms as well as their inverses in the $C^k$-topology. As in the case of a general compact manifold the following definition connects to it.

\begin{definition}
	\begin{enumerate}
		\item A sequence of $C^{\infty}$ diffeomorphisms is called convergent in $\diff^{\infty}$ if it converges in $\diff^k(\mathbb{S}^1)$ for every $k \in \mathbb{N}$.
		\item On $\diff^{\infty}(\mathbb{S}^1)$ we declare the following metric
		\begin{equation*}
			d_{\infty}\left(f,g\right) = \sum^{\infty}_{k=1} \frac{d_k\left(f,g\right)}{2^k \cdot \left(1 + d_k\left(f,g\right)\right)}.
		\end{equation*}
	\end{enumerate}
\end{definition}

The space $\mathcal{H}^{\infty}$ equipped with the metric $d_{\infty}$ is complete and separable. Hence, $\mathcal{H}^{\infty}$ is a Polish space. We refer to \cite[chapter I.2]{He} for a thorough description of spaces of circle diffeomorphisms.

\subsection{Some tools from Descriptive Set Theory} \label{subsec:Descriptive}
The main tool is the idea of a
reduction for equivalence relations. See \cite{Fsurvey} and \cite{Gao} for further information.
\begin{definition}
	Let $X$ and $Y$ be Polish spaces and $E\subseteq X \times X$, $F\subseteq Y \times Y$ be equivalence relations.
	\begin{center}
		A function $f:X\to Y$ \emph{reduces} $E$ to $F$ \\
		if and only if \\
		for all $x_1,x_2\in X$: $x_1Ex_2$ if and only if $f(x_1)Ff(x_2)$. 
	\end{center}
	
	Such a function $f$ is called a Borel (respectively, continuous) reduction
	if $f$ is a Borel (respectively, continuous) function. 
	
	We use the notation $E \precsim_{\mathcal{B}} F$ (respectively, $E \precsim_{C} F$).
\end{definition}

$E$ being reducible to $F$ can be interpreted as saying that $F$
is at least as complicated as $E$.

In this work, we want to exclude the existence of complete numerical invariants. Thus, we consider the \emph{equality equivalence relation} $=_Y\subseteq Y \times Y$ on a Polish space $Y$. Since for every Polish space $Y$ there is a Borel injection $g : Y \to \R\setminus \Q$, we can extend a Borel reduction $f$ of any equivalence relation $E \subseteq X \times X$ to $(Y,=_Y)$ to a Borel reduction $g \circ f$ of $(X,E)$ to $(\R,=_R)$. Thus, we can assume that Borel reductions to any $=_Y$ can be changed to Borel reductions to equality on the real numbers.

In the next step, we define the equivalence relation $E_0$ which is an important tool to exclude the existence of complete numerical invariants.
 
\begin{definition}
	Let $E_0$ be the equivalence relation on $\{0,1\}^{\N}$ defined by setting 
	\[
	(a_n)_{n\in \N} E_0 (b_n)_{n\in \N} \text{ if and only if there is $N\in \N$ such that $a_m = b_m$ for all $m>N$.} 
	\]
\end{definition}

The significance of equivalence relation $E_0$ for excluding complete numerical invariants comes from the so-called \emph{Harrington-Kechris-Louveau dichotomy} (sometimes also refered to as Glimm-Effros dichotomy). It states that for any Borel equivalence relation $E$ on a Polish space $X$ exactly one of the following alternatives holds: either $E \precsim_{\mathcal{B}} =_{\R}$ or $E_0 \precsim_{\mathcal{B}} E$.

To exclude the existence of complete numerical invariants, we use the following direction from the Harrington-Kechris-Louveau dichotomy (see \cite[Theorem 36]{Fsurvey} or \cite[Proposition 6.1.7]{Gao}).

\begin{theorem} \label{thm:Crit}
	Suppose that $E$ is an equivalence relation on an uncountable Polish space $X$ and $E_0 \precsim_{\mathcal{B}} E$. Then $E\not\precsim_{\mathcal{B}} =_{\R}$.
\end{theorem}

\begin{proof}
	For the reader's convenience we include Feldman's argument from \cite{Fe74}.
	
    Since $E_0 \precsim_{\mathcal{B}} E$, there is a Borel map $\psi:\{0,1\}^{\mathbb{N}}\to X$
    such that for $(a_{n})_{n\in\mathbb{N}},(b_{n})_{n\in\mathbb{N}}\in\{0,1\}^{\mathbb{N}},$
    \begin{equation}
    	\psi\left((a_{n})_{n\in\mathbb{N}}\right)E\psi\left((b_{n})_{n\in\mathbb{N}}\right)\text{ if and only if }a_{n}=b_{n}\text{ for all but at most finitely many  }n.\label{eq:equivalence}
    \end{equation}
    We want to show that there does not exist a complete numerical Borel invariant for $E$, that
    is, there does not exist a Borel function $\varphi:X\to\mathbb{R}$
    such that $xEy$ if and only if $\varphi(x)=\varphi(y)$.
    
    Suppose such a Borel function $\varphi$
    exists. Then for $(a_{n})_{n\in\mathbb{N}},(b_{n})_{n\in\mathbb{N}}\in\{0,1\}^{\mathbb{N}}$,
    $\varphi\circ\psi\left((a_{n})_{n\in\mathbb{N}}\right)=\varphi\circ\psi\left((b_{n})_{n\in\mathbb{N}}\right)$
    if and only if $a_{n}=b_{n}$ for all but at most finitely many $n$.
    Thus, for any Borel set $B\subset\mathbb{R}$, $(\varphi\circ\psi)^{-1}(B)$
    is invariant under maps that consist of a permutation of finitely
    many terms in the elements of $(\varphi\circ\psi)^{-1}(B)$. If we
    endow $\{0,1\}^{\mathbb{N}}$ with the $(\frac{1}{2},\frac{1}{2})$-Bernoulli measure, then by the Hewitt-Savage zero-one law, any such
    set $(\varphi\circ\psi)^{-1}(B)$ has measure $0$ or $1$. Therefore
    $\varphi\circ\psi$ is almost everywhere constant on $\{0,1\}^{\mathbb{N}}$.
    Hence, there are uncountably many elements of $\{0,1\}^{\mathbb{N}}$
    that are mapped to the same element of $\mathbb{R}.$ This contradicts
    \eqref{eq:equivalence}. 
\end{proof}

\section{Strategy of proof} \label{sec:strategy}

Let $\mathcal{H}^{\infty}_{\alpha}$ be the collection of orientation-preserving $C^{\infty}$-diffeomorphisms with rotation number $\alpha \in \mathbb{S}^1$. As we show below, our Main Theorem \ref{theo:A} follows from the subsequent stronger Proposition which is proved in Section \ref{sec:ProofProp}.

\begin{prop} \label{prop:Main}
	Let  $\alpha \in \mathbb{S}^1$ be a Liouville number. There is a  continuous one-to-one map 
	\[
	\Psi: \{0,1\}^{\N} \to \mathcal{H}^{\infty}_{\alpha}
	\]
	such that for any two sequences $\a=(a_n)_{n\in \N}$ and $\b =(b_n)_{n\in \N}$ the following properties hold:
	\begin{enumerate}
		\item[(R1)] If there is $N \in \N$ such that $a_n =b_n$ for every $n \geq N$, then the $C^{\infty}$-diffeomorphisms $\Psi(\a)$ and $\Psi(\b)$ are $C^{\infty}$-conjugate.
		\item[(R2)] If there are infinitely many $n \in \N$ with $a_n \neq b_n$, then the $C^{\infty}$-diffeomorphisms $\Psi(\a)$ and $\Psi(\b)$ are not Hölder-conjugate.
	\end{enumerate}
\end{prop}
	
	Using the notions from descriptive set theory in Section \ref{subsec:Descriptive}, Proposition \ref{prop:Main} implies the following statement.
	
	\begin{cor}
		Let $\alpha \in \mathbb{S}^1$ be a Liouville number and $\mathcal{D}$ be the collection of orientation-preserving circle homeomorphisms with regularity $(D)$, where $(D)$ could be any degree of regularity from Hölder to $C^{\infty}$. Then there is a continuous reduction from $E_0$ to the $\mathcal{D}$-conjugacy relation of diffeomorphisms in $\mathcal{H}^{\infty}_{\alpha}$. 
	\end{cor}

\begin{proof}
	If $\a E_0 \b$, then the $C^{\infty}$-diffeomorphisms $\Psi(\a)$ and $\Psi(\b)$ are $C^{\infty}$-conjugated by part (1) of Proposition \ref{prop:Main}. Hence, they are $\mathcal{D}$-conjugate. \\
	If $\a$ and $\b$ are not equivalent in $E_0$, then there are infinitely many $n \in \N$ with $a_n \neq b_n$. Thus, the $C^{\infty}$-diffeomorphisms $\Psi(\a)$ and $\Psi(\b)$ are not Hölder-conjugate by part (2) of Proposition \ref{prop:Main} and, hence, they are not $\mathcal{D}$-conjugate. \\
	Altogether, the $C^{\infty}$-diffeomorphisms $\Psi(\a)$ and $\Psi(\b)$ are $\mathcal{D}$-conjugate if and only if $\a E_0 \b$. This shows that the continuous map $\Psi: \{0,1\}^{\N} \to \mathcal{H}^{\infty}_{\alpha}$ is a reduction from $E_0$ to the $\mathcal{D}$-conjugacy relation of orientation-preserving $C^{\infty}$-diffeomorphisms of the circle. 
\end{proof}

Then we apply Theorem \ref{thm:Crit} to conclude the following result which in particular implies our Main Theorem \ref{theo:A}.  

\begin{theorem}
	Let  $\alpha \in \mathbb{S}^1$ be a Liouville number and $\mathcal{D}$ be the collection of orientation-preserving circle homeomorphisms with regularity $(D)$, where $(D)$ could be any degree of regularity from Hölder to $C^{\infty}$. Then there is no complete numerical Borel invariant for the $\mathcal{D}$-conjugacy relation of  diffeomorphisms in $\mathcal{H}^{\infty}_{\alpha}$.
\end{theorem}
	
	\section{Construction of the reduction} \label{sec:constr}
	Let  $\alpha \in \mathbb{S}^1$ be a Liouville number. For every sequence $\a=(a_n)_{n\in \N}$ we construct a diffeomorphism $T_{\a} \in \mathcal{H}^{\infty}_{\alpha}$ by an inductive construction process as the limit of a sequence $(T_{\a , n})_{n \in \N_0}$ of $C^{\infty}$-diffeomorphisms of the circle defined by
	\begin{equation} \label{eq:AbC}
		T_{\a,n}=H_{\a ,n} \circ R_{\alpha_{n+1}} \circ H^{-1}_{\a , n} \text{ for } n \in \N_0.
	\end{equation} 
   Here, we take $H_{\a , 0} = \operatorname{id}$ and for $n \in \N$ we successively construct rational numbers $\alpha_n = \frac{p_n}{q_n}$ as well as conjugation maps $H_{\a , n} = H_{\a , n-1} \circ h_{\a , n}$ with a $C^{\infty}$-diffeomorphism $h_{\a , n}$ satisfying 
	\begin{equation} \label{eq:equiv}
	h_{\a , n} \circ R_{1/q_n} = R_{1/q_n} \circ h_{\a , n}. 
	\end{equation}
	
	In Subsection \ref{subsec:conj} we present the construction of our conjugation map $h_{\a , n}$ which depends on entry $a_n$ in the sequence $\a=(a_n)_{n\in \N}$. This map will allow explicit norm estimates that we use in Subsection \ref{subsec:conv} to prove convergence of the sequence $(T_{\a ,n})_{n \in \N}$ in the space $\mathcal{H}^{\infty}_{\alpha}$ of orientation-preserving $C^{\infty}$-diffeomorphisms with the prescribed rotation number $\alpha$.
	
	\subsection{Construction of conjugation maps} \label{subsec:conj}
	We use a $C^{\infty}$-function $\psi: \mathbb{R} \rightarrow [0,1]$ satisfying $\psi\left( (-\infty, 0] \right) = {0}$, $\psi\left( [ 1, \infty ) \right) = {1}$ and $\psi$ is strictly monotone increasing on $[0,1]$. For any $t \in (0,1/4)$ we define a strictly increasing $C^{\infty}$ diffeomorphism $\hat{h}_t:[0,1] \to [0,1]$ as follows: $\hat{h}_t\left(x\right) =$
	\begin{equation*}
		 \begin{cases} 
			x & \textit{ if } x \in \left[0, t\right], \\
			\left(1-\psi(t^{-1}(x-t))\right) \cdot x + \psi(t^{-1}(x-t)) \cdot \left(t \cdot \left(x-\frac{1}{4}\right)+3t\right) & \textit{ if } x \in \left[t, 2t\right], \\
			t \cdot \left(x-\frac{1}{4}\right)+3t & \textit{ if } x \in \left[2t, \frac{1-t+8t^2}{2}\right], \\
			\left(1-\psi(t^{-2}(x-\frac{1-t+8t^2}{2}))\right) \cdot \left(t \left(x-\frac{1}{4}\right)+3t\right) & \\
			+ \psi(t^{-2}(x-\frac{1-t+8t^2}{2})) \cdot \left(t^{-1} \left( x - \frac{1}{2}\right)+\frac{1}{2}\right) & \textit{ if } x \in \left[\frac{1-t+8t^2}{2},\frac{1-t+10t^2}{2}\right], \\
			t^{-1} \cdot \left( x - \frac{1}{2}\right)+\frac{1}{2} & \textit{ if } x \in \left[ \frac{1-t+10t^2}{2}, \frac{1+t-10t^2}{2} \right], \\
			\left(1-\psi(t^{-2}(x-\frac{1+t-10t^2}{2}))\right) \cdot \left(t^{-1} \left( x - \frac{1}{2}\right)+\frac{1}{2}\right) & \\
			+ \psi(t^{-2}(x-\frac{1+t-10t^2}{2})) \cdot \left(t  \left(x-\frac{3}{4}\right)+1-3t\right) & \textit{ if } x \in \left[\frac{1+t-10t^2}{2},\frac{1+t-8t^2}{2}\right], \\
			t \cdot \left(x-\frac{3}{4}\right)+1-3t & \textit{ if } x \in \left[ \frac{1+t-8t^2}{2} , 1-2t \right], \\
			\left(1-\psi(t^{-1}(x-1+2t))\right) \cdot \left(t  \left(x-\frac{3}{4}\right)+1-3t\right)+ \psi(t^{-1}(x-1+2t))  x & \textit{ if } x \in \left[1-2t,1-t\right], \\
			x & \textit{ if } x \in \left[1-t,1 \right].
		\end{cases}
	\end{equation*}

In particular, $\hat{h}_t$ coincides with the identity in the neighborhood of the boundary. See Figure \ref{fig:1} for a visualisation of such a map. 
\begin{figure}[h]
	\begin{center}
		\includegraphics[scale=0.75]{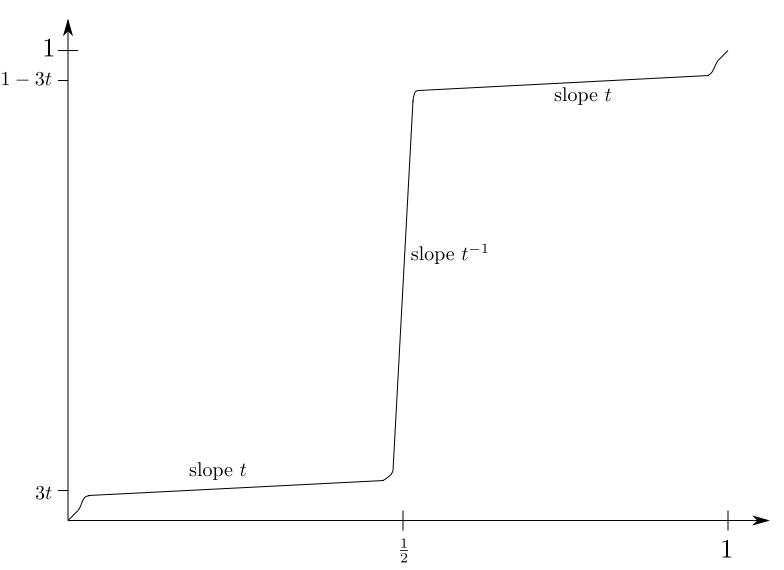}
		\caption{Qualitative shape of the function $\hat{h}_t$}
		\label{fig:1}
	\end{center}
\end{figure}
We use this explicit form of the map $\hat{h}_t$ to obtain norm estimates as a power of $q_n$ in \eqref{eq:normh}. We also note that $\hat{h}^2_t$ has slope $t^{-2}$ on 
\begin{equation} \label{eq:Slope}
\hat{I}_t \coloneqq \left[ \frac{1}{2}-\frac{t^2}{4}, \frac{1}{2}+ \frac{t^2}{4} \right].
\end{equation}

Since $\hat{h}_t$ is a strictly increasing $C^{\infty}$ function, it has a strictly increasing smooth inverse function satisfying 
\begin{equation*}
	\hat{h}^{-1}_t\left(x\right) = \begin{cases} 
		x & \textit{ if } x \in \left[0,t\right], \\
		t^{-1} \cdot \left(x-3t\right)+\frac{1}{4} & \textit{ if } x \in \left[3t-\frac{t-8t^2}{4}, 3t+\frac{t-2t^2}{4}\right], \\
		t \cdot \left( x - \frac{1}{2}\right)+\frac{1}{2} & \textit{ if } x \in \left[ 5t , 1-5t \right], \\
		t^{-1} \cdot \left(x-1+3t\right)+\frac{3}{4} & \textit{ if } x \in \left[ 1-3t-\frac{t-2t^2}{4}, 1-3t+\frac{t-8t^2}{4} \right], \\
		x & \textit{ if } x \in \left[1-t,1 \right].
	\end{cases}
\end{equation*}
Hence, $\hat{h}^{-2}_t$ has slope $t^{-2}$ on 
\begin{equation}\label{eq:SlopeInv}
\hat{J}_t \coloneqq \left[3t-\frac{t}{4}+3t^2-\frac{t^2}{8}, 3t-\frac{t}{4}+3t^2+ \frac{t^2}{8}\right] \cup \left[ 1-3t+\frac{t}{4}-3t^2-\frac{t^2}{8}, 1-3t+\frac{t}{4}-3t^2+ \frac{t^2}{8} \right].
\end{equation}

	For any $n \in \N$ we take 
	\begin{equation} \label{eq:t}
		t_n = q^{1-2n}_n
	\end{equation}
	and define $\tilde{h}_n = \hat{h}_{t_n}$. Let $h_{q_n}$ be the lift of $\tilde{h}_n$ by the cyclic $q_n$-fold covering map $\pi_{q_n}$ such that $\operatorname{Fix}(h_{q_n}) \neq \emptyset$. Finally, we define
	\[
	h_{\a , n} = \begin{cases}
		h_{q_n} & \text{ if } a_n =0, \\
		h^{-1}_{q_n} & \text{ if } a_n =1.
	\end{cases}
	\]
	Clearly, the commutativity condition \eqref{eq:equiv} is satisfied. We also observe for any $r \in \N$ that 
	\begin{equation} \label{eq:normh}
		\vertiii{h_{\a , n}}_r \leq C_{n,r} \cdot q^{N(n,r)}_n
	\end{equation}
	with some positive integer $N(n,r)$ and some constant $C_{n,r}$ that are independent of $q_n$ and $\a \in \{0,1\}^{\N}$. 
	With the aid of this estimate and a generalized chain rule we are able to prove an estimate on the norms of $H_{\a , n} = H_{\a , n-1} \circ h_{\a , n}$.
	
	\begin{lemma}\label{lem:normH}
		For every $k \in \N$ we have
		\[
		\vertiii{H_{\a ,n}}_k \leq C(n,H_{\a , n-1},k) \cdot q^{k\cdot N(n,k)}_n,
		\]
		where $C(n,H_{\a , n-1},k)$ is a constant depending solely on $n$, $H_{\a , n-1}$, and $k$. Since $H_{\a , n-1}$ is independent of
		$q_n$ in particular, the same is true for $C(n,H_{\a , n-1},k)$.
	\end{lemma}

\begin{proof}
	By the Fa\`a di Bruno's formula we get
	\[
	\norm{D^k\left(H_{\a , n-1} \circ h_{\a , n}\right)} \leq C(k) \cdot \norm{H_{\a , n-1}}_k \cdot \norm{h_{\a , n}}^k_k
	\]
	and
	\[
	\norm{D^k \left(h^{-1}_{\a , n} \circ H^{-1}_{\a , n-1}\right)} \leq C(k) \cdot \norm{h^{-1}_{\a , n}}_k \cdot \norm{H^{-1}_{\a , n-1}}^k_k
	\]
	with a constant $C(k)$ solely depending on $k$. Using \eqref{eq:normh} these estimates imply
	\[
	\vertiii{H_{\a ,n}}_k \leq C(k) \cdot \vertiii{H_{\a , n-1}}^k_k \cdot \vertiii{h_{\a , n}}^k_k \leq C(n,H_{\a , n-1},k) \cdot q^{kN(n,k)}_n,
	\]
	where $C(n,H_{\a , n-1},k)$ is a constant depending solely on $n$, $H_{\a , n-1}$, and $k$.
\end{proof}

	\subsection{Convergence of the sequence $(T_{\a ,n})_{n \in \N}$} \label{subsec:conv}
	We start our proof of convergence of the sequence $(T_{\a ,n})_{n \in \N}$ in $\mathcal{H}^{\infty}_{\alpha}$ by stating the following lemma based upon the mean value theorem and the chain rule.
\begin{lemma}[\cite{FS}, Lemma 5.6] \label{lem:konj}
	Let $k \in \mathbb{N}_0$ and $h$ be a $C^{\infty}$-diffeomorphism. Then we get for every $\alpha,\beta \in \mathbb{R}$:
	\begin{equation*}
		d_k\left(h \circ R_{\alpha} \circ h^{-1}, h \circ R_{\beta} \circ h^{-1}\right) \leq C_k \cdot \vertiii{h}^{k+1}_{k+1} \cdot \left| \alpha - \beta \right|,
	\end{equation*}
	where the constant $C_k$ depends solely on $k$. In particular $C_0 = 1$.
\end{lemma}

Since $\alpha$ is a Liouville number, the proof of convergence of $(T_{\a ,n})_{n \in \N}$ in $\mathcal{H}^{\infty}_{\alpha}$ is similar to the proof of the general criterion in \cite[Lemma 5.7]{FS}. We additionally ensure that the same sequence $(\alpha_n)_{n \in \N}$ yields convergence for all $\a=(a_n)_{n\in \N}\in \{0,1\}^{\N}$. This is used in the third part of the lemma which will guarantee continuity of our reduction $\Psi$.
\begin{lemma} \label{lem:conv}
	There is a sequence $(\alpha_n)_{n \in \N}$ of rational numbers $\alpha_n =\frac{p_n}{q_n}$ with
	\begin{equation} \label{eq:qGrowth}
		q_{n+1} \geq q^{2n}_n \ \text{ for every } n\in \N
	\end{equation}
	such that the following properties hold:
	\begin{enumerate}
		\item For every sequence $\a \in \{0,1\}^{\N}$ the sequence of diffeomorphisms $(T_{\a ,n})_{n \in \N}$ defined by \eqref{eq:AbC} converges in the $\diff^{\infty}$-topology to a smooth diffeomorphism $T_{\a}$. 
		\item Also the sequence of diffeomorphisms $\hat{T}_{\a ,n}=H_{\a ,n} \circ R_{\alpha} \circ H^{-1}_{\a ,n} \in \mathcal{H}^{\infty}_{\alpha}$ converges to $T_{\a}$ in the $\diff^{\infty}$-topology. Hence, $T_{\a} \in \mathcal{H}^{\infty}_{\alpha}$.
		\item For every $\varepsilon>0$ there is $N \in \N$ such that for all sequences $\a=(a_n)_{n\in \N}, \b=(b_n)_{n\in \N} \in \{0,1\}^{\N}$ with $a_n=b_n$ for all $n \leq N$ we have
		\[
		d_{\infty}(T_{\a},T_{\b}) <\varepsilon.
		\]
	\end{enumerate}
\end{lemma}

\begin{proof}
	\begin{enumerate}
		\item According to condition \eqref{eq:equiv} of our construction we have $h_{\a , n} \circ R_{\alpha_n} = R_{\alpha_n} \circ h_{\a , n}$ and, hence,
		\begin{align*}
			T_{\a , n-1} & = H_{\a , n-1} \circ R_{\alpha_n} \circ H^{-1}_{\a , n-1} = H_{\a , n-1} \circ R_{\alpha_n} \circ h_{\a , n} \circ h^{-1}_{\a , n} \circ H^{-1}_{ \a , n-1}  \\
			& = H_{\a , n-1} \circ h_{\a , n} \circ R_{\alpha_n} \circ h^{-1}_{\a , n} \circ H^{-1}_{\a , n-1} = H_{\a , n} \circ R_{\alpha_n} \circ H^{-1}_{\a , n}
		\end{align*}
		for every sequence $\a \in \{0,1\}^{\N}$. Applying Lemma \ref{lem:konj} we obtain for every $k,n \in \mathbb{N}$:
		\begin{equation}  \label{eq:est1}
			\begin{split}
				d_k\left(T_{\a , n}, T_{\a , n-1}\right) & = d_k\left(H_{\a , n} \circ R_{\alpha_{n+1}} \circ T^{-1}_{\a , n}, H_{\a , n} \circ R_{\alpha_n} \circ H^{-1}_{\a , n}\right) \\
				& \leq C_k \cdot \vertiii{H_{\a , n}}^{k+1}_{k+1} \cdot \left| \alpha_{n+1} - \alpha_n \right|.
			\end{split}
		\end{equation}
		We assume $\left|\alpha - \alpha_n \right| \stackrel{n \rightarrow \infty}{\longrightarrow} 0$ monotonically. Using the triangle inequality we obtain $\left| \alpha_{n+1} - \alpha_n \right| \leq \left| \alpha_{n+1} - \alpha \right| + \left| \alpha - \alpha_n \right| \leq 2 \cdot \left| \alpha - \alpha_n \right|$ and therefore equation \eqref{eq:est1} becomes:
		\begin{equation*}
			d_k\left(T_{\a , n}, T_{\a , n-1}\right) \leq C_k \cdot \vertiii{H_{\a , n}}^{k+1}_{k+1} \cdot 2 \cdot \left| \alpha - \alpha_{n} \right|.
		\end{equation*}
		To estimate the norm of $H_{\a , n}$ we use Lemma \ref{lem:normH}. Here, we note that $H_{\a , n-1}$ depends on the sequence entries $a_1, \dots , a_{n-1}$ only. Hence, 
		\[
		\hat{C}_n \coloneqq \max_{\a \in \{0,1\}^{\N}} C\left(n,H_{\a , n-1},n+1\right)
		\] 
		is well-defined, where $C(n,H_{\a , n-1},k)$ are the constants from Lemma \ref{lem:normH}. Then Lemma \ref{lem:normH} yields
		\[
		\vertiii{H_{\a , n}}_{n +1} \leq \hat{C}_n \cdot q^{(n+1) \cdot N(n,n+1)}_n
		\]
		for every sequence $\a \in \{0,1\}^{\N}$.
		
		Since $\alpha$ is a Liouville number, there are $q_n > q^{2 \cdot (n-1)}_{n-1}$ and $p_n \in \N$ such that
		\[
		\abs{\alpha - \frac{p_n}{q_n}} < \min \left( \abs{\alpha-\alpha_{n-1}}, \, \frac{1}{2 \cdot n^2 \cdot C_n \cdot \left(\hat{C}_n \cdot q^{(n+1) \cdot N(n,n+1)}_n\right)^{n+1}} \right).
		\] 
		In particular, this implies condition \eqref{eq:qGrowth} and 
		\begin{equation} \label{eq:alpha}
		\abs{\alpha - \frac{p_n}{q_n}} < \frac{1}{2 \cdot n^2 \cdot C_{n} \cdot \vertiii{H_{\a , n}}^{n+1}_{n +1} }
		\end{equation}
		for every sequence $\a \in \{0,1\}^{\N}$. It follows for every sequence $\a \in \{0,1\}^{\N}$ and every $k \leq n$ that 
		\begin{equation} \label{eq:est4}
			\begin{split}
				d_k\left(T_{\a , n},T_{\a , n-1}\right) & \leq d_{n}\left(T_{\a , n},T_{\a , n-1}\right) \leq C_{n} \cdot \vertiii{H_{\a , n}}^{n+1}_{n+1} \cdot 2 \cdot \left| \alpha - \alpha_{n} \right| \\ 
				& \leq C_{n} \cdot \vertiii{H_{\a , n}}^{n+1}_{n+1} \cdot 2 \cdot \frac{1}{2 \cdot n^2 \cdot C_{n} \cdot \vertiii{H_{\a , n}}^{n+1}_{n +1} } \leq \frac{1}{n^2}.
			\end{split}
		\end{equation}
		In the next step we show, that for every sequence $\a \in \{0,1\}^{\N}$ and for arbitrary $k \in \mathbb{N}$ our sequence $\left(T_{\a , n}\right)_{n \in \mathbb{N}}$ is a Cauchy sequence in $\diff^k(\mathbb{S}^1)$, that is, $\lim_{n,m\rightarrow \infty} d_k\left(T_{\a , n}, T_{\a , m}\right) = 0$. For this purpose, we calculate
		\begin{equation} \label{eq:est2}
			\lim_{n \rightarrow \infty} d_k\left(T_{\a , n}, T_{\a , m}\right) \leq \lim_{n \rightarrow \infty} \sum^{n}_{i=m+1} d_k\left(T_{\a , i}, T_{\a , i-1}\right) = \sum^{\infty}_{i=m+1}d_k\left(T_{\a , i}, T_{\a , i-1}\right).
		\end{equation}
		We consider the limit process $ m \rightarrow \infty$, i.e. we can assume $k\leq m$, and obtain from equations \eqref{eq:est4} and \eqref{eq:est2} that
		\begin{equation*}
			\lim_{n,m \rightarrow \infty} d_k\left(T_{\a , n}, T_{\a , m}\right) \leq \lim_{m \rightarrow \infty} \sum^{\infty}_{i=m+1} \frac{1}{i^2} = 0.
		\end{equation*}
		Since $\diff^k(\mathbb{S}^1)$ is complete, the sequence $\left(T_{\a , n}\right)_{n \in \mathbb{N}}$ converges consequently in $\diff^k(\mathbb{S}^1)$ for every $k \in \mathbb{N}$ and for every sequence $\a \in \{0,1\}^{\N}$. Thus, the sequence converges in $\diff^{\infty}(\mathbb{S}^1)$ by definition.
		\item To show $\hat{T}_{\a ,n } \rightarrow T_{\a}$ in $\diff^{\infty}(\mathbb{S}^1)$ we compute with the aid of Lemma \ref{lem:konj} and equation \eqref{eq:alpha} that for every $n \in \mathbb{N}$ and $k \leq n$ we have
		\begin{align*}
			d_k\left(T_{\a , n}, \hat{T}_{\a , n}\right) & \leq d_{n}\left(H_{\a , n} \circ R_{\alpha_{n+1}} \circ H^{-1}_{\a , n}, H_{\a , n} \circ R_{\alpha} \circ H^{-1}_{\a , n}\right) \\
			& \leq C_{n} \cdot \vertiii{H_{\a , n}}^{n +1}_{n+1} \cdot \left| \alpha_{n+1} - \alpha \right| \leq C_{n} \cdot \vertiii{H_{\a , n}}^{n +1}_{n+1} \cdot \left| \alpha_{n} - \alpha \right| \\
			& \leq C_{n} \cdot \vertiii{H_{\a , n}}^{n +1}_{n+1} \cdot \frac{1}{2 \cdot n^2 \cdot C_{n} \cdot \vertiii{H_{\a , n}}^{n+1}_{n +1} } \leq \frac{1}{n^2}.
		\end{align*}
		This yields $\hat{T}_{\a ,n } \rightarrow T_{\a}$ in $\diff^k(\mathbb{S}^1)$ for any $k \in \N$ and $\lim_{n\rightarrow \infty}d_{\infty}\left(\hat{T}_{\a,n},T_{\a}\right) = 0$ as asserted. 
		
		Furthermore, we recall that the rotation number $\tau: \mathcal{H} \to \mathbb{S}^1$ is a continuous map in the $C^0$-topology \cite[Proposition 11.1.6]{HK}. Since $\tau(\hat{T}_{\a , n})= \alpha$ for all $n \in \N$, we also obtain $\tau (T_{\a}) = \alpha$.
		\item Since $T_{\a ,n}$ converges to $T_{\a}$ in $\diff^{\infty}$ by the first part of the lemma, there is $N_1 \in \N$ such that $d_{\infty}(T_{\a ,n},T_{\a}) <\frac{\varepsilon}{2}$ for all $n \geq N_1$. Analogously, there is $N_2 \in \N$ such that $d_{\infty}(T_{\b ,n},T_{\b}) <\frac{\varepsilon}{2}$ for all $n \geq N_2$. Let $N\coloneqq \max(N_1,N_2)$ and suppose that $a_n=b_n$ for all $n \leq N$. In particular, this implies $T_{\a , n} = T_{\b , n}$ for all $n \leq N$. Altogether we have
		\begin{align*}
			d_{\infty}(T_{\a},T_{\b}) & \leq d_{\infty}(T_{\a},T_{\a , N}) + d_{\infty}(T_{\a , N},T_{\b , N}) + d_{\infty}(T_{\b , N},T_{\b}) \\
			& = d_{\infty}(T_{\a},T_{\a , N}) + d_{\infty}(T_{\b , N},T_{\b}) \\
			& < \frac{\varepsilon}{2} + \frac{\varepsilon}{2} = \varepsilon.
		\end{align*}
	\end{enumerate}	
\end{proof}

	\section{Proof of Proposition \ref{prop:Main}} \label{sec:ProofProp}

In combination with Lemma \ref{lem:ConjRot} the following result shows that the sequence $(H_{\a , n})_{n \in \N}$ converges to the unique conjugating homeomorphism~$H_{\a}$ between the rotation $R_{\alpha}$ and $T_{\a}$ with $H_{\a}(0)=0$.

	\begin{lemma} \label{lem:convH}
		For every $\a \in \{0,1\}^{\N}$ the sequence $(H_{\a , n})_{n \in \N}$ converges in $d_0$ to a homeomorphism~$H_{\a}$ with $H_{\a}(0)=0$.
	\end{lemma}

\begin{proof}
	Since $h_{q_n}$ is constructed as a lift of the homeomorphism $h_n$ by the $q_n$-fold cyclic covering map, we obtain 
	\[
	\norm{h_{q_n} - \operatorname{id}}_0 \leq q^{-1}_n \ \text{ as well as } \ \norm{h^{-1}_{q_n} - \operatorname{id}}_0 \leq q^{-1}_n.
	\]
	We use this to estimate
	\begin{equation*}
		\norm{H^{-1}_{\a ,n} - H^{-1}_{\a , n-1}}_0 = \norm{ (h^{-1}_{\a , n} - \operatorname{id})\circ H^{-1}_{\a , n-1}}_0 = \norm{h^{-1}_{\a , n}-\operatorname{id}}_0 \leq q^{-1}_n.
	\end{equation*} 
	Then we obtain
	\begin{equation*}
		\norm{ H^{-1}_{\a , n+k} - H^{-1}_{\a , n-1} }_0 \leq \sum^{k}_{\ell=0} \norm{H^{-1}_{\a , n+\ell}-H^{-1}_{\a , n+\ell-1}}_0 \leq \sum^k_{\ell=0} q^{-1}_{n+\ell}.
	\end{equation*}
	Since $\sum^{\infty}_{n=1}q^{-1}_n < \infty$ by Lemma \ref{lem:conv}, $\left(H^{-1}_{\a , n}\right)_{n \in \mathbb{N}}$ is a Cauchy sequence. This shows the uniform convergence of $\left(H^{-1}_{\a ,n}\right)_{n \in \mathbb{N}}$ to a continuous map $H^{-1}_{\a}: \mathbb{S}^1 \rightarrow \mathbb{S}^1$. Additionally, $H^{-1}_{\a}$ is monotone as a uniform limit of homeomorphisms. Due to the second part of Lemma \ref{lem:conv} the sequence of $C^{\infty}$-diffeomorphisms $\hat{T}_{\a , n}=H_{\a , n} \circ R_{\alpha} \circ H^{-1}_{\a , n}$ converges to a diffeomorphism $T_{\a}$ in the $\text{Diff}^{\infty}$-topology. Hereby, we obtain $H^{-1}_{\a} \circ T_{\a} = R_{\alpha} \circ H^{-1}_{\a}$. Altogether, we conclude that  $H^{-1}_{\a}$ is a homeomorphism.
	
	In the next step, we observe that
	\[
	\norm{H_{\a , n} -H_{\a}}_0 = \norm{\operatorname{id}-H_{\a}\circ H^{-1}_{\a,n}}_0 = \norm{H_{\a} \circ (H^{-1}_{\a}-H^{-1}_{\a , n})}_0 \to 0 \text{ as } n \to \infty
	\]
	by uniform continuity of $H_{\a}$. Altogether, $(H_{\a , n})_{n \in \N}$ converges in $d_0$ to the homeomorphism~$H_{\a}$. 
	
	For all $n \in \N$ we have $h_{\a , n}(0)=0$ and, hence, $H_{\a,n}(0)=0$. Thus, we also get $H_{\a}(0)=0$.
\end{proof}

We are now ready to conclude the proof of the main proposition.

\begin{proof}[Proof of Proposition \ref{prop:Main}]
	We define the map $\Psi$ by setting $\Psi(\a)=T_{\a}$, where $T_{\a} \in \mathcal{H}^{\infty}$ is the limit of the sequence of diffeomorphisms $(T_{\a ,n})_{n \in \N}$ constructed by \eqref{eq:AbC}. This limit exists by the first part of Lemma \ref{lem:conv}. The second part of Lemma \ref{lem:conv} and the convergence of $(H_{\a , n})_{n \in \N}$ to a homeomorphism~$H_{\a}$ yield that $T_{\a}$ has rotation number $\alpha$. Hence, $\Psi(\a)=T_{\a} \in \mathcal{H}^{\infty}_{\alpha}$ for all $\a \in \{0,1\}^{\N}$. Furthermore, the third part of Lemma \ref{lem:conv} proves the continuity of $\Psi$.
	
	Finally, we check that $\Psi$ satisfies properties (R1) and (R2). For this purpose, we examine the regularity of conjugacy between $T_{\a}$ and $T_{\b}$. Here, we note as a consequence of Lemma \ref{lem:ConjUniq} and Lemma \ref{lem:convH} that the sequence $(H_{\b , n} \circ H^{-1}_{\a , n})_{n \in \N}$ converges to the unique conjugating homeomorphism $G_{\a , \b} \coloneqq H_{\b} \circ H^{-1}_{\a}$ between $T_{\a}$ and $T_{\b}$ with $G_{\a , \b}(0)=0$.
	\begin{enumerate}
		\item[(R1)] Let $\a=(a_n)_{n\in \N}$ and $\b =(b_n)_{n\in \N}$ be sequences such that there is $N \in \N$ with $a_n =b_n$ for every $n \geq N$. Then
		\begin{equation*}
		H_{\b , n} \circ H^{-1}_{\a , n}  = H_{\b,N} \circ h_{\b, N+1} \circ \dots \circ  h_{\b , n} \circ h^{-1}_{\a, n} \circ \dots \circ h^{-1}_{\a , N+1} \circ H^{-1}_{\a , N} = H_{\b,N}H^{-1}_{\a,N}
		\end{equation*}
		for every $n \geq N$. Hence,		
		\[
		G_{\a , \b} = H_{\b}\circ H^{-1}_{\a} = H_{\b , N} \circ H^{-1}_{\a , N} \in \mathcal{H}^{\infty}
		\]
		is the unique conjugating homeomorphism between $T_{\a}$ and $T_{\b}$ satisfying $G_{\a , \b}(0)=0$, that is, $T_{\a}$ and $T_{\b}$ are $C^{\infty}$-conjugate. 
		\item[(R2)] Let $\a=(a_n)_{n\in \N}$ and $\b =(b_n)_{n\in \N}$ be sequences such that there are infinitely many $n \in \N$ with $a_n \neq b_n$. We denote the set of these indices with $a_n \neq b_n$ by $\mathcal{N}$. Furthermore, we recall the definition of intervals $\hat{I}_t$ and $\hat{J}_t$ from equations \eqref{eq:Slope} and \eqref{eq:SlopeInv}. Using them, we introduce the following sets $K_n$ for each $n \in \N$:
		\begin{itemize}
			\item If $a_n = 1$, let $k_{n,1} \in \N$ be the smallest integer and $k_{n,2} \in \N$ be the largest integer such that
			\[
			\left[\frac{k_{n,1}}{q_{n+1}}, \frac{k_{n,2}}{q_{n+1}} \right] \subseteq \left[ \frac{1}{2q_n}-\frac{t_n}{4q_n}, \frac{1}{2q_n}+ \frac{t_n}{4q_n} \right].
			\] 
			Then we define 
			\[
			K_n \coloneqq \bigcup^{q_n -1}_{\ell=0} \left[\frac{\ell}{q_n}+\frac{k_{n,1}}{q_{n+1}}, \frac{\ell}{q_n}+\frac{k_{n,2}}{q_{n+1}} \right]
			\]
			in order to have $h_{\a , n}(K_n) \subseteq \pi^{-1}_{q_n}(\hat{I}_{t_n})$.
			\item If $a_n = 0$, let $k_{n,1} \in \N$ be the smallest integer and $k_{n,2} \in \N$ be the largest integer such that
			\[
			\left[\frac{k_{n,1}}{q_{n+1}}, \frac{k_{n,2}}{q_{n+1}} \right] \subseteq \left[ \frac{3t_n}{q_n}-\frac{t_n}{8q_n}, \frac{3t_n}{q_n}+ \frac{t_n}{8q_n} \right].
			\] 
			Then we define 
			\[
			K_n \coloneqq \bigcup^{q_n -1}_{\ell=0} \left[\frac{\ell}{q_n}+\frac{k_{n,1}}{q_{n+1}}, \frac{\ell}{q_n}+\frac{k_{n,2}}{q_{n+1}} \right]
			\]
			in order to have $h_{\a , n}(K_n) \subseteq \pi^{-1}_{q_n}(\hat{J}_{t_n})$.
		\end{itemize}
	In both cases, we call such an interval
	$$\left[\frac{\ell}{q_n}+\frac{k_{n,1}}{q_{n+1}}, \frac{\ell}{q_n}+\frac{k_{n,2}}{q_{n+1}} \right]$$ a component of $K_n$. These choices quarantee that $h_{\b , n} \circ h^{-1}_{\a, n}|_{h_{\a , n}(K_n)}$ has slope $t^{-2}_n$ for $n \in \mathcal{N}$. We are now ready to show that $G_{\a , \b}$ is not $d$-Hölder continuous for any $d \in (0,1)$. 
	
	For every $i \in \N$ and any component $\tilde{K}_i$ of $K_i$ there is a component $\tilde{K}_{i+1}$ of $K_{i+1}$ such that $h_{\a , i+1}\left(\tilde{K}_{i+1}\right)\subseteq \tilde{K}_i$. This proves the existence of a component $\tilde{K}_n$ of $K_n$ such that $H_{\a , n}|_{\tilde{K}_n}$ is an affine transformation of slope $\prod^n_{i=1}t_i$ and $H_{\b , n} \circ H^{-1}_{\a , n}|_{H_{\a , n}(\tilde{K}_n)}$ is an affine transformation of slope 
	\begin{equation*}
		\prod_{i \in \mathcal{N}, i\leq n } t^{-2}_i.
	\end{equation*}
	In the following, we denote $\tilde{K}_n = [x',y']$. In particular, we have $\abs{y'-x'}\geq 8^{-1}t_nq^{-1}_n$. Moreover, we define $x,y \in \mathbb{S}^1$ by $x=H_{\a , n}(x')$ and $y=H_{\a , n}(y')$. We also introduce the notation $G^{(n+1)}_{\a , \b}=H^{-1}_{\b , n}G_{\a , \b}H_{\a , n}$. We observe that $G^{(n+1)}_{\a , \b}$ is the uniform limit of $h_{\b , n+1} \circ \dots \circ h_{\b , n+m} \circ h^{-1}_{\a, n+m}\circ \dots \circ h^{-1}_{\a , n+1}$ as $m \rightarrow \infty$. Since the maps $h_i$ are $q^{-1}_i$-cyclic, we have $x',y' \in \operatorname{Fix}(h_i)$ for all $i>n$. Thus, we obtain
	\begin{align*}
	\abs{G_{\a , \b}(x)-G_{\a, \b}(y)}& =\abs{ H_{\b , n}\left(G^{(n+1)}_{\a , \b}(x')\right) - H_{\b , n}\left(G^{(n+1)}_{\a , \b}(y')\right) } = \abs{ H_{\b , n} (x') - H_{\b , n}(y') } \\ 
	& =\prod_{i \in \mathcal{N}, i\leq n } t^{-2}_i \cdot \abs{x-y}. 
\end{align*}
	
	Suppose $n \in \mathcal{N}$. Then we obtain the following estimate using the properties of the domain $K_n$ as well as equation \eqref{eq:t}: 
\begin{align*}
	\frac{\abs{G_{\a , \b}(x)-G_{\a, \b}(y)}}{\abs{x-y}^{\frac{1}{n}}} & = \prod_{i \in \mathcal{N}, i\leq n } t^{-2}_i \cdot \abs{x-y}^{1- \frac{1}{n}} \\
	& = \prod_{i \in \mathcal{N}, i\leq n } t^{-2}_i \cdot \left( \prod^n_{i=1} t_i \right)^{1- \frac{1}{n}} \cdot \abs{x'-y'}^{1- \frac{1}{n}} \\
	& \geq \prod_{i \in \mathcal{N}, i\leq n } t^{-2}_i \cdot \left( \prod^n_{i=1} t_i \right)^{1- \frac{1}{n}} \cdot \left( 8^{-1}t_nq^{-1}_n \right)^{1- \frac{1}{n}} \\
	& \geq 8^{-1} \cdot \prod_{i \in \mathcal{N}, i < n } t^{-2}_i \cdot \left( \prod^{n-1}_{i=1} t_i \right)^{1- \frac{1}{n}} \cdot \left( t_n  \right)^{-\frac{2}{n}} \cdot \left(q_n\right)^{\frac{1}{n}-1} \\
	& = 8^{-1} \cdot \prod_{i \in \mathcal{N}, i < n } t^{-2}_i \cdot \left( \prod^{n-1}_{i=1} t_i \right)^{1- \frac{1}{n}} \cdot \left(q_n\right)^{3-\frac{1}{n}}.
\end{align*}
We also recall the condition $q_{i+1} \geq q^{2 \cdot i}_{i}$ for every $i \in \N$ from \eqref{eq:qGrowth}. This implies 
\[
\prod^{n-1}_{i=1}q^{2i-1}_i \leq q_n
\] 
and, hence, 
\[
\prod^{n-1}_{i=1} t_i = \left( \prod^{n-1}_{i=1}q^{2i-1}_i \right)^{-1} \geq q^{-1}_n,
\]
where we used the definition of $t_i$ from equation \eqref{eq:t}. Altogether, we obtain
\[
\frac{\abs{G_{\a , \b}(x)-G_{\a, \b}(y)} }{\abs{x-y}^{\frac{1}{n}}} \geq q_n.
\]
Since there are infinitely many $n \in \mathcal{N}$, we conclude that $G_{\a , \b}$ cannot be $d$-H\"older for any $d \in (0,1)$. This finishes the proof of property (R2).
	\end{enumerate}
\end{proof}

\subsection{Absolute continuous and singular invariant measures}\label{subsec:Abs}
It is a well-known fact that every circle homeomorphism $T$ with irrational rotation number is uniquely ergodic (see \cite[Theorem 11.2.9]{HK}). We denote the unique $T$-invariant probability measure by $\mu_T$. If $T=H \circ R_{\alpha} \circ H^{-1}$ with unique $H \in \mathcal{H}$ satisfying $H(0)=0$, then $\mu_T$ is given by $\mu_T = H_{\ast}m$, where $m$ is the Lebesgue measure on $\mathbb{S}^1$. 

Moreover, \cite[Theorem VII.1.4]{He} shows that $T \in \mathcal{H}^2$ with $\tau(T) \notin \Q$  is ergodic with respect to Lebesgue measure $m$. Hence, any $T$-invariant
Borel subset is either $m$-null or $m$-conull. Together with the uniqueness of $\mu_T$, this implies that either $\mu_T$ is equivalent to $m$ or singular to $m$. 

If $\mu_T$ is equivalent to $m$, then the unique conjugating homeomorphism $H$ with $H(0)=0$ maps any Lebesgue null set to a null set. In this case, the conjugacy is called \emph{absolutely continuous}. If $\mu_T$ is singular to $m$, then conjugacy $H$ maps some Lebesgue null set to a conull set. In this case, the conjugacy is called \emph{singular}. 

In this section, we show that both cases can be realized in the setting of our Main Theorem~\ref{theo:A}. We start by showing that the conjugation maps as defined in Section \ref{subsec:conj} result in a singular conjugacy $H_{\a}$. Here, we follow an approach similar to \cite[section 3]{M} and \cite[Lemma 1.1]{KuCircle}.
\begin{lemma}
	$H_{\a}$ is singular.
\end{lemma}

\begin{proof}
Let $n \in \N$ and suppose that $a_n =0$. Then we let $\ell_{n,1}, \ell_{n,3} \in \N$ be the smallest positive integers and $\ell_{n,2}, \ell_{n,4} \in \N$ be the largest positive integers such that
\[
\left[ \frac{\ell_{n,1}}{q_{n+1}},\frac{\ell_{n,2}}{q_{n+1}}\right] \subseteq \left[\frac{2t_n}{q_n}, \frac{1-t_n+8t^2_n}{2q_n}\right] \ \text{ and } \ \left[ \frac{\ell_{n,3}}{q_{n+1}},\frac{\ell_{n,4}}{q_{n+1}}\right] \subseteq \left[ \frac{1+t_n-8t^2_n}{2q_n} , \frac{1-2t_n}{q_n} \right].
\]
Then we define the set 
\begin{equation*}
	L_n= \bigcup^{q_n -1}_{k=0}\left[ \frac{k}{q_n}+ \frac{\ell_{n,1}}{q_{n+1}},\frac{k}{q_n}+\frac{\ell_{n,2}}{q_{n+1}}\right] \cup  \left[\frac{k}{q_n}+ \frac{\ell_{n,3}}{q_{n+1}}, \frac{k}{q_n}+\frac{\ell_{n,4}}{q_{n+1}}\right].
\end{equation*}
By construction, $h_{\a , n}|_{L_n}$ is an affine transformation with slope $t_n$. Furthermore, we introduce the set $M_n$ defined by
\begin{equation*}
	\bigcup^{q_n -1}_{k=0}\left[ \frac{k+3t_n}{q_n}- \frac{t_n-8t^2_n}{4q_{n}},\frac{k+3t_n}{q_n}+\frac{t_n-2t^2_n}{4q_{n}}\right] \cup  \left[\frac{k+1-3t_n}{q_n}- \frac{t_n-2t^2_n}{4q_{n}}, \frac{k+1-3t_n}{q_n}+\frac{t_n-8t^2_n}{4q_{n}}\right].
\end{equation*}
Due to $t_n < 1$ we have $M_n \subset L_n$. Additionally, we observe $h_{\a , n} \left(L_n\right) \subseteq M_n$. By direct computation, we also obtain $m\left(L_n\right) \geq 1-6t_n$ and $m\left(M_n\right) \leq t_n$, where $m$ stands for the Lebesgue measure.

In an analogous manner, we handle the case of $a_n=1$ and define sets $L_n$ and $M_n$ with $m\left(L_n\right) \geq 1-6t_n$ and $m\left(M_n\right) \leq t_n$ in such a way that $h_{\a , n}|_{L_n}$ is an affine transformation with slope $t_n$ as well as $h_{\a , n} \left(L_n\right) \subseteq M_n$.

	Let $C_n = \bigcap^n_{i=1} L_i$ and $C = \bigcap^{\infty}_{i=1} L_i$. Since each component of $L_i$ is an interval with endpoints in $(q^{-1}_{i+1}\Z)/\Z$, we have
	\begin{equation*}
		m \left( C_n \right) \geq \prod^n_{i=1} \left(1-6t_i \right).
	\end{equation*}
	By definition of $t_n$ from equation \eqref{eq:t} we obtain 
	\begin{equation*}
		\sum^{\infty}_{n=1} t_n = \sum^{\infty}_{n=1} q^{1-2n}_n< \infty
	\end{equation*}
	and, hence, $\mu_T(H_{\a}(C)) = m(C)>0$. 
	
	By construction of $h_{\a , j}$ with the aid of the $q_j$-fold covering map, $\left(q^{-1}_{n+1}\mathbb{Z}\right) / \mathbb{Z}$ is pointwise fixed under $h_{\a , j}$, $j>n$. Since any component of $L_{n}$ is an interval with endpoints in $\left(q^{-1}_{n+1}\mathbb{Z}\right) / \mathbb{Z}$, we get $h_{\a , j}(L_n)=L_n$ for $j>n$. Then we obtain for any $j>n$ that
	\begin{equation*}
		H_{\a , j}(L_n) = H_{\a , n}(L_n)= H_{\a , n-1}h_{\a , n}(L_n) \subseteq H_{\a , n-1}(M_n).
	\end{equation*}
	This yields $L_n \subseteq H^{-1}_{\a , j}H_{\a , n-1}(M_n)$, where $H_{\a , 0} = \operatorname{id}$. Thus, the uniform limit $H^{-1}_{\a}$ of $H^{-1}_{\a , j}$ satisfies $L_n \subseteq H^{-1}_{\a}\left( H_{\a , n-1}(M_n)\right)$ for any $n \in \mathbb{N}$. Hereby, we observe
	\begin{equation*}
		H_{\a}\left(C_n\right) = H_{\a} \left(\bigcap^ n_{i=1} L_i \right) = \bigcap^ n_{i=1}  H_{\a} \left(L_i \right) \subseteq \bigcap^ n_{i=1}  H_{\a , i-1} \left(M_i \right).
	\end{equation*}
	In order to have $H_{\a , i-1}\left(A\right) \subset H_{\a , i-2}\left(M_{i-1}\right)$ for a set $A\subset \mathbb{S}^1$, this set $A$ has to satisfy $h_{\a , i-1}\left(A\right) \subset M_{i-1}$ which implies the condition $A \subset L_{i-1}$. Since $m(M_i) \leq t_i$ and the slope of $h_{i-1}|_{L_{i-1}}$ is equal to $t_{i-1}$, this yields $m\left(H_{\a}\left(C_n\right)\right) \leq \prod^{n-1}_{i=1} t_i$ which converges to $0$ as $n \rightarrow \infty$. Therefore, $m(H_{\a}(C))= 0$. 
	
	Altogether, we conclude that $\mu_T$ is not equivalent to $m$ because $\mu_T(H_{\a}(C))>0$ and $m(H_{\a}(C))=0$. Hence, $H_{\a}$ is a singular map. 
\end{proof}

In the next step, we present some modifications to the construction of the conjugation maps $h_{\a , n}$ which will allow us to produce absolute continuous conjugacies. For this purpose, we recall that the map $\hat{h}_t:[0,1] \to [0,1]$ from Section \ref{subsec:conj} coincides with the identity in a neighborhood of the boundary. Using the maps $C_n: \left[0, \frac{1}{2^{n+1}}\right] \rightarrow [0,1]$, $C_n(x)=2^{n+1} \cdot x$, we construct the orientation-preserving circle diffeomorphism $\breve{h}_{t_n}$ as follows:
\begin{equation*}
	\breve{h}_{t_n}(x) = \begin{cases}
		C^{-1}_n \circ \hat{h}_{t_n} \circ C_n(x) & \text{ if } x \in \left[0, \frac{1}{2^{n+1}}\right] \\
		x & \text{ if } x \in \left[ \frac{1}{2^{n+1}}, 1\right]
	\end{cases} ,
\end{equation*}
where we define the number $t_n$ as in the previous section. 

Let $\overline{h}_{q_n}$ be the lift of $\breve{h}_{t_n}$ by the cyclic $q_n$-fold covering map $\pi_{q_n}$ such that $\operatorname{Fix}(\overline{h}_{q_n}) \neq \emptyset$. Finally, we define the conjugation maps $\overline{h}_{\a , n}$ using the map $\overline{h}_{q_n}$ instead of $h_{q_n}$. By the same reasoning as above, one can show that the sequence $\left(\overline{H}_{\a , n} \right)_{n \in \N}$ converges to a homeomorphism $\overline{H}_{\a}$ and that the limit diffeomorphisms $\overline{T}_{\a} = \overline{H}_{\a} \circ R_{\alpha} \circ \overline{H}^{-1}_{\a}$ satisfy Proposition \ref{prop:Main} as well.

We now prove the absolute continuity of $H$ by the same method as in \cite[section 4]{M} and \cite[Lemma 2.7]{KuCircle}.

\begin{lemma}
	$H_{\a}$ is absolutely continuous.
\end{lemma}

\begin{proof}
	We introduce the sets
	\begin{equation*}
		\breve{L}_n = \left[ 2^{-(n+1)},1\right] \ \text{ and } \ \overline{L}_n = \pi^{-1}_{q_n}\left(\breve{L}_n\right).
	\end{equation*}
	According to our construction $\overline{h}_{\a , n}$ is the identity on $\overline{L}_n$. Let $X= \bigcap^{\infty}_{n=1} \overline{L}_n$. Then we have
	\begin{equation*}
		m(X) \geq 1 - \sum^{\infty}_{n=1} m\left(\mathbb{S}^1 \setminus \overline{L}_n \right) = 1 - \sum^{\infty}_{n=1} 2^{-(n+1)} = \frac{1}{2}.
	\end{equation*}
	Since $\overline{H}_{\a}$ is the identity on the positive measure set $X$, we have for any Borel set $B$ that $\mu_{\overline{T}_{\a}}(B \cap X) = m(B \cap X)$ and $\mu_{\overline{T}_{\a}}(X)=m(X)  > 0$. 
	
	Assume that $\mu_{\overline{T}_{\a}}$ is not equivalent to $m$. Then $\mu_{\overline{T}_{\a}}$ is singular to $m$ and there is a Borel set $B \subseteq \mathbb{S}^1$ such that $m(B)=1$ and $\mu_{\overline{T}_{\a}}(B)=0$. But then we obtain the contradiction $m(B \cap X) = m(X) >0$ but $\mu_{\overline{T}_{\a}}(B \cap X) \leq \mu_{\overline{T}_{\a}}(B) =0$. Hence, $H_{\a}$ is absolutely continuous.
\end{proof}
	
	\section{Higher rank actions} \label{sec:Higher}
	In this final section we turn to actions by $\Z^d$, $d\geq 2$, on the circle. We obtain a generalization of our Main Theorem \ref{theo:A}.
	
	\begin{maintheorem}
		Let $(D)$ be some degree of regularity from Hölder to $C^{\infty}$ and $\mathcal{D}$ be the collection of orientation-preserving circle homeomorphisms with regularity $(D)$. Then there is no complete numerical Borel invariant for $\mathcal{D}$-conjugacy of free $\Z^d$ actions by orientation-preserving $C^{\infty}$ diffeomorphisms of the circle. 
	\end{maintheorem}

By our strategy of proof from Section \ref{sec:strategy}, this theorem follows from the subsequent criterion.

\begin{prop} \label{prop:Main2}
	Let $\mathcal{A}$ be the space of free $\Z^d$ actions by orientation-preserving $C^{\infty}$ diffeomorphisms of the circle. There is a  continuous one-to-one map 
	\[
	\Phi: \{0,1\}^{\N} \to \mathcal{A}
	\]
	such that for any two sequences $\a=(a_n)_{n\in \N}$ and $\b =(b_n)_{n\in \N}$ the following properties hold:
	\begin{enumerate}
		\item[(R1)] If there is $N \in \N$ such that $a_n =b_n$ for every $n \geq N$, then the $C^{\infty}$-actions $\Phi(\a)$ and $\Phi(\b)$ are $C^{\infty}$-conjugate.
		\item[(R2)] If there are infinitely many $n \in \N$ with $a_n \neq b_n$, then the $C^{\infty}$-actions $\Phi(\a)$ and $\Phi(\b)$ are not Hölder-conjugate.
	\end{enumerate}
\end{prop}

In order to build the smooth $\Z^d$ actions, we construct generating diffeomorphisms by a slight modification to the constructions in Section \ref{sec:constr}. In particular, we have to arrange for commutativity of the generators and freeness of the action. For the latter one, the following number-theoretical lemma from \cite{Ha} will prove useful.

\begin{lemma}[\cite{Ha}, Theorem 2.1] \label{lem:Indep}
	Let $d \in \Z^+$, $\epsilon >0$ and $(a_{i,n})_{n\in \N}$ for $i=1,\dots , d$ be sequences of positive integers such that
	\begin{enumerate}
		\item $a_{1,n}$ divides $a_{1,n+1}$ and $\frac{a_{1,n+1}}{a_{1,n}} \geq 2^{(d+1)^{n-1}}$
		\item $b_{i,n} < 2^{(d+1)^{n-(\sqrt{2}+\epsilon)\sqrt{n}}}$ for $i=1,\dots , d$
		\item $\lim_{n\to \infty} \frac{a_{i,n}b_{j,n}}{b_{i,n}a_{j,n}}=0$ for all $i,j \in \{1,\dots , d\}$, $i>j$
		\item $a_{i,n}2^{-(d+1)^{n-(\sqrt{2}+\epsilon)\sqrt{n}}}<a_{1,n} < a_{i,n}2^{(d+1)^{n-(\sqrt{2}+\epsilon)\sqrt{n}}}$ for $i=1,\dots , d$
	\end{enumerate}
hold for every sufficiently large $n \in \Z^+$. Then the numbers $\sum^{\infty}_{n=1} \frac{b_{1,n}}{a_{1,n}}, \dots , \sum^{\infty}_{n=1} \frac{b_{d,n}}{a_{d,n}}$ and $1$ are linearly independent over the integers.
\end{lemma}

\begin{proof}[Proof of Proposition \ref{prop:Main2}]
	We construct the $d$ geneators $T^{(i)}_{\a}$, $i=1,\dots , d$, of the $\Z^d$ action $\Phi(\a)$ as limits of AbC diffeomorphisms $T^{(i)}_{\a , n}=H_{\a , n} \circ R_{\alpha^{(i)}_{n+1}} \circ H^{-1}_{\a , n}$, where the conjugation maps $H_{\a,n}=H_{\a , n-1} \circ h_{\a , n}$ are constructed as in Section \ref{subsec:conj} and $(\alpha^{(i)}_n)_{n \in \N}$ are sequences of rational numbers $\alpha^{(i)}_n = \frac{p^{(i)}_n}{q_n}$ with $p^{(i)}_n$ and $q_n$ relatively prime.
	
	For a start, we let $H_{\a , 0}=\operatorname{id}$, $q_1$ be a power of $2$ and for each $i=1,\dots , d$ we let $p^{(i)}_1$ be an odd integer. In the induction step from $n-1$ to $n$ we construct the conjugation map $h_{\a , n}$ as in Section \ref{subsec:conj} using the number $q_n$. Then we choose $l_n$ as a sufficiently large  power of $2$ such that
	\begin{equation}\label{eq:l}
		l_n > 4^{dn}\cdot C_n \cdot \vertiii{H_{\a , n}}^{n+1}_{n+1}
	\end{equation}
for all sequences $\a \in \{0,1\}^{\N}$, where $C_n$ is the constant from Lemma \ref{lem:konj}. We proceed by defining the rational numbers
\begin{equation} \label{eq:dAlpha}
	\alpha^{(i)}_{n+1} \coloneqq \alpha^{(i)}_n + \frac{3^{(i-1)\cdot n}}{l_n q^2_n} = \frac{p^{(i)}_nl_nq^2_n+3^{(i-1)\cdot n}}{l_n q^2_n}
\end{equation}
for each $i=1,\dots , d$. Since $l_n$ and $q_n$ are powers of $2$, we note that $q_{n+1}=l_nq^2_n$ is a power of $2$ and $p^{(i)}_{n+1}$ is relatively prime to $q_{n+1}$. Furthermore, we apply Lemma \ref{lem:konj} to get the following estimate for each $i=1,\dots , d$ and every $\a \in \{0,1\}^{\N}$:
\begin{align*}
	d_n\left( T^{(i)}_{\a , n} , T^{(i)}_{\a , n-1} \right) 
	& = d_n \left( H_{\a , n} \circ R_{\alpha^{(i)}_{n+1}} \circ H^{-1}_{\a , n},   H_{\a , n-1} \circ R_{\alpha^{(i)}_{n}} \circ H^{-1}_{\a , n-1} \right) \\
	& = d_n \left( H_{\a , n} \circ R_{\alpha^{(i)}_{n+1}} \circ H^{-1}_{\a , n},   H_{\a , n-1} \circ R_{\frac{p^{(i)}_n}{q_n}} \circ h_{\a , n} \circ h^{-1}_{\a , n} \circ H^{-1}_{\a , n-1} \right) \\
	& = d_n \left( H_{\a , n} \circ R_{\alpha^{(i)}_{n+1}} \circ H^{-1}_{\a , n},   H_{\a , n-1} \circ h_{\a , n} \circ R_{\frac{p^{(i)}_n}{q_n}}  \circ h^{-1}_{\a , n} \circ H^{-1}_{\a , n-1} \right) \\
	& = d_n \left( H_{\a , n} \circ R_{\alpha^{(i)}_{n+1}} \circ H^{-1}_{\a , n},   H_{\a , n} \circ R_{\alpha^{(i)}_{n}} \circ H^{-1}_{\a , n} \right) \\
	& \leq C_n \cdot \vertiii{H_{\a , n}}^{n+1}_{n+1} \cdot \abs{\alpha^{(i)}_{n+1} - \alpha^{(i)}_n} \\
	& < \frac{3^{(d-1)n}}{4^{dn} \cdot q^2_n},
\end{align*}
where we used equations \eqref{eq:l} and \eqref{eq:dAlpha} in the last step. As in the proof of Lemma \ref{lem:conv}, this estimate allows us to conclude convergence of the sequence $(T^{(i)}_{\a , n})_{n \in \N}$ to a $C^{\infty}$-diffeomorphism $T^{(i)}_{\a}$. By Lemma \ref{lem:convH} we again obtain convergence of $(H_{\a , n})_{n \in \N}$ to a homeomorphism $H_{\a}$. Hence, for each $i=1,\dots , d$ we get $T^{(i)}_{\a} = H_{\a} \circ R_{\alpha^{(i)}} \circ H^{-1}_{\a}$, where $\alpha^{(i)}$ is the limit of the sequence $(\alpha^{(i)}_n)_{n \in \N}$. In particular, we have $T^{(i)}_{\a} \circ T^{(j)}_{\a} = T^{(j)}_{\a} \circ T^{(i)}_{\a}$ for all $i,j \in \{1,\dots , d\}$. Thus, the $C^{\infty}$-diffeomormphisms $T^{(1)}_{\a}, \dots , T^{(d)}_{\a}$ generate a smooth $\Z^d$ action. The properties (R1) and (R2) follow from the corresponding properties in Proposition \ref{prop:Main}.

Finally, we apply Lemma \ref{lem:Indep} with $a_{i,n}=q_{n+1}$ and $b_{i,n}= 3^{(i-1)n}$ to verify that the numbers $\alpha^{(1)},\dots , \alpha^{(d)}$ and $1$ are linearly independent over the integers. This yields that the action $\Phi(\a)$ generated by $T^{(1)}_{\a}, \dots , T^{(d)}_{\a}$ is a free action. 
\end{proof}

	\paragraph{Acknowledgments:} The author would like to thank Matthew Foreman, Marlies Gerber, Federico Rodriguez Hertz, and Jean-Paul Thouvenot for helpful comments. In particular, Foreman's preprint \cite{Fsurvey} inspired this project and Rodriguez Hertz suggested to look into consequences for higher rank actions. The author also thanks organizers and participants of the Banff workshop ``Interactions between Descriptive Set Theory and Smooth Dynamics'' for a cooperative atmosphere and helpful discussions.

\end{document}